\documentclass[11pt, oneside]{amsart}   	
\usepackage[
backend=biber,
style=numeric,
backref=true,
]{biblatex}
\usepackage{csquotes}

\usepackage[margin = 0.75in]{geometry}                		
\usepackage{graphicx}				

\DeclareFontEncoding{FMS}{}{}
\DeclareFontSubstitution{FMS}{futm}{m}{n}
\DeclareFontEncoding{FMX}{}{}
\DeclareFontSubstitution{FMX}{futm}{m}{n}
\DeclareSymbolFont{fouriersymbols}{FMS}{futm}{m}{n}
\DeclareSymbolFont{fourierlargesymbols}{FMX}{futm}{m}{n}
\DeclareMathDelimiter{\VERT}{\mathord}{fouriersymbols}{152}{fourierlargesymbols}{147}

\usepackage[english]{babel}
\usepackage{amsthm}
\usepackage{bbm}	
\usepackage{url}

\usepackage[
colorlinks=true]{hyperref} 
\usepackage[nameinlink,capitalise]{cleveref}
\usepackage{ifthen}
\usepackage{tikz}
\usetikzlibrary{decorations.pathreplacing}
\usepackage{tikz-cd}

\allowdisplaybreaks

\usepackage{blindtext}
\usepackage{scrextend}
\addtokomafont{labelinglabel}{\sffamily}

\usepackage{amssymb}
\usepackage{amsmath}

\newcommand{\ph}{\varphi}
\newcommand{\NN}{\mathbb{N}}

\newcommand{\PP}{\mathbb{P}}
\newcommand{\dd}{\mathop{}\!\mathrm{d}}
\makeatletter
\newcommand{\subalign}[1]{%
  \vcenter{%
    \Let@ \restore@math@cr \default@tag
    \baselineskip\fontdimen10 \scriptfont\tw@
    \advance\baselineskip\fontdimen12 \scriptfont\tw@
    \lineskip\thr@@\fontdimen8 \scriptfont\thr@@
    \lineskiplimit\lineskip
    \ialign{\hfil$\m@th\scriptstyle##$&$\m@th\scriptstyle{}##$\hfil\crcr
      #1\crcr
    }%
  }%
}
\makeatother

\newtheorem{theorem}{Theorem}[section]
\newtheorem{corollary}[theorem]{Corollary}
\newtheorem{lemma}[theorem]{Lemma}
\newtheorem{proposition}[theorem]{Proposition}
\theoremstyle{definition}
\newtheorem{definition}[theorem]{Definition}

\newtheorem{remark}[theorem]{Remark}

\title{On Robin's Inequality and the Kaneko-Lagarias Inequality}
\author{Idris Assani}
\address{Department of Mathematics, The University of North Carolina at Chapel Hill, 120 E Cameron Avenue, CB 3250
Chapel Hill, NC 27599-3250, USA}
\email{assani@email.unc.edu}
\urladdr{https://idrisassani.web.unc.edu/}

\author{Aiden Chester}
\address{The University of North Carolina at Chapel Hill, 120 E Cameron Avenue, CB 3250
Chapel Hill, NC 27599-3250, USA}
\email{achester@unc.edu}

\author{Alex Paschal}
\address{The University of North Carolina at Chapel Hill, 120 E Cameron Avenue, CB 3250
Chapel Hill, NC 27599-3250, USA}
\email{ampasch@unc.edu}


\begin{document}
\begin{abstract}
    
    We provide new, elementary proofs that Robin's inequality and the Lagarias inequality hold for almost every number, including all numbers not divisible by one of the prime numbers $2$, $3$,  $5$; all primorials; given $k$ a natural number, all sufficiently large numbers of the form $2^kn$ for $n\ge1$ odd; and all $21$-free integers. Additionally, we prove that the Kaneko-Lagarias inequality holds for all natural numbers if and only if it holds for all superabundant numbers.

\end{abstract}
\maketitle

\tableofcontents
\section{Introduction}

We define the \emph{sum of divisors function} and \emph{Euler's totient function} as
\begin{equation*}
    \sigma(n)=\sum_{d\mid n}d,\quad\ph(n)=n\prod_{p\mid n}\bigg(1-\frac 1p\bigg)
\end{equation*}
respectively, where the product is taken over primes $p$ which divide $n$. It has long been known that these functions have connections to the Riemann hypothesis (RH). Robin's inequality \cite{Robin} states that RH holds if and only if, for all $n>5040$,
\begin{equation}\label{Robin's inequality}
    \sigma(n)<e^\gamma \log(\log (n)),
\end{equation}
where $\gamma\approx .57721\dots$ denotes the Euler-Mascheroni constant.

We briefly survey some known results concerning families of natural numbers which satisfy Robin's inequality. A number is $k$-free if all the powers in its prime factorization are $<k$. In \cite{Choieetal}, it is proven that all all odd integers $>9$ satisfy Robin's inequality, as do all $5$-free integers. This was extended to $7$-free numbers in \cite{SolePlanat}, $11$-free numbers in \cite{BroughanTrudgian}, $20$-free numbers in \cite{MorrillandPlatt}, and finally $21$-free numbers in \cite{Axler23}. We reprove the result of Axler \cite[Theorem 1]{Axler23} using more elementary methods in \cref{21-free numbers}: while Axler's proof relies on estimates from \cite{SolePlanat} and combinatorial prime counting algorithms, our proof uses only arithmetic manipulations and a sharper bound from \cite{AxlerNicolas} than the one used in \cite[Theorem 15]{RosserShoenfeld}.

Similarly, \cref{Robin's inequality for numbers not divisible by p_j}, which states that Robin's inequality holds for numbers not divisible by one of the primes $2$, $3$, $5$, is implied by the results of \cite{Hertlein16}, but is proven using more elementary methods. The proof in \cite{Hertlein16}, which works with a number's \emph{$p$-adic order}, relies on an algorithm from \cite{AFJ07}, whereas ours is derived only from arithmetic manipulations.

The final class of results concerning Robin's inequality is density results. The first such result is by Robin \cite{Robin}, who proved that Robin's inequality holds for all square-free (that is, $2$-free) numbers. In \cite[p. 46]{Tenenbaum}, it is shown that the logarithmic density of non-square-free integers is $\frac 12-\frac{2}{\pi^2}\approx .2973...$. Similarly, \cref{Robin's inequality for numbers not divisible by p_j} shows that Robin's inequality holds for a set of logarithmic density $\frac{29}{30}$. Wójtowicz \cite{Wojtowicz07} was the first to show that Robin's inequality holds on a set of density $1$. We prove this using different methods in \cref{almost every number}: our proof, again, relies mostly on arithmetic manipulations, as opposed to the ``deep'' results of Ford and Luca-Pomerance. It is worth noting that according to \cite{Luca}, for $x>7!$,
\begin{equation}
    \#\{n \le x \mid \sigma(n) \ge e^{\gamma}n\log(\log(n)) \}=x^{O(\frac{1}{\log(\log(x))})},
\end{equation}
so the density of counterexamples to Robin's inequality up to large but finite $x$ is also quite small.

Another well-studied inequality which is equivalent to RH is the Lagarias inequality. Denote by $H_n$ the $n$-th \emph{harmonic number}; that is, $H_n=1+\frac 12+\cdots+\frac 1n$. The Lagarias inequality \cite{Lagarias} states that, for all $n\geq 1$,
\begin{equation*}
    \sigma(n)<H_n+\exp(H_n)\log(H_n).
\end{equation*}
It turns out that the $H_n$ term on the RHS is negligible in the sense that the following inequality, which we name the Kaneko-Lagarias inequality (see the acknowledgements in \cite{Lagarias}), is also equivalent to RH: for all $n>60$,
\begin{equation*}
    \sigma(n)<\exp(H_n)\log(H_n).
\end{equation*}
We note that similar alternative inequalities have been introduced \cite{WashingtonYang21}.

A number is \emph{superabundant} if $m<n$ implies $\sigma(m)/m<\sigma(n)/n$. By dividing both sides of \eqref{Robin's inequality} by $n$ and noting that the RHS is monotone increasing, one immediately sees that Robin's inequality holds if and only if it holds for superabundant numbers (this observation was made in \cite{AkbaryFriggstad}). One would like to say the same for the Lagarias inequality and the Kaneko-Lagarias inequality, but the picture is more complicated since monotonicity is harder to prove. Nevertheless, we prove in \cref{LLI theorem} that the Kaneko-Lagarias inequality holds if and only if it holds for superabundant numbers. We would like to extend this result to the Lagarias inequality in future work.

The layout of our paper is as follows. In \cref{Robin's inequality sec}, we consider Robin's inequality. Each subsection corresponds to a result which we prove, and is labeled as such. We note that we use a unified method throughout the section, reobtaining some results which were found using varied methods. In \cref{Lagarias ineq sec}, we focus on the Lagarias inequality. We introduce the Kaneko-Lagarias inequality and prove the contents of the last paragraph.

\section{Robin's Inequality}\label{Robin's inequality sec}

\subsection{Sufficiently big numbers not divisible by one of the prime numbers 2,3,5}

Let $p_1=2$, $p_2=3$, etc. be an enumeration of the prime numbers which we denote by $\PP$. Fix $j\in\NN$ and let $q_1<q_2<\cdots<q_k$ be some prime numbers distinct from $p_j$. Given $\alpha_1,\alpha_2,\dots,\alpha_k\in\NN$, let $n=q_1^{\alpha_1}q_2^{\alpha_2}\cdots q_k^{\alpha_k}$.

\begin{lemma}\label{sigma(n)/n<n/phi(n)} 
    We have
    \begin{equation}\label{first bound on sigma(n)/n}
        \frac{\sigma(n)}{n}<\prod_{\ell=1}^k\frac{q_\ell}{q_\ell -1}\leq\prod_{\subalign{\ell=&1,\dots,j-1\\&j+1,\dots,k+1}}\frac{p_\ell}{p_\ell-1}=\frac{n}{\ph(n)}.
    \end{equation}
\end{lemma}

\begin{proof}
    The first inequality follows from the fact that for any $p\in\PP$ and $\alpha\in\NN$
    \begin{equation}
        \frac{\sigma(p^\alpha)}{p^\alpha}=\frac{p-\frac{1}{p^\alpha}}{p-1}\nearrow\frac{p}{p-1}\text{ as }\alpha\to\infty.
    \end{equation}
    The second inequality follows from the fact that $p_i\leq q_i$ for all $1\leq i\leq k$.
\end{proof}

Note that
\begin{equation}\label{n/phi(n)=A(k)B(k)} 
    \frac{n}{\ph(n)}=\left(\prod_{\subalign{\ell=&1,\dots,j-1\\&j+1,\dots,k+1}}\frac{p_\ell+1}{p_\ell}\right)\left(\prod_{\subalign{\ell=&1,\dots,j-1\\&j+1,\dots,k+1}}\frac{p_\ell^2}{p_\ell^2-1}\right)=:A(k)B(k).
\end{equation}
We can bound $A(k)$ as follows:
\begin{equation}
    \log(A(k))=\sum_{\subalign{\ell=&1,\dots,j-1\\&j+1,\dots,k+1}}\log\left(1+\frac{1}{p_\ell}\right)\leq\sum_{\subalign{\ell=&1,\dots,j-1\\&j+1,\dots,k+1}}\frac{1}{p_\ell}=\left(\sum_{\ell=1}^{k+1}\frac{1}{p_\ell}\right)-\frac{1}{p_j}
\end{equation}
and
\begin{equation}
    \sum_{\ell=1}^{k+1}\frac{1}{p_\ell}\leq\log(\log(p_{k+1}))+c_1+\frac{5}{\log(p_{k+1})},
\end{equation}
where $c_1\approx .261497$ by Theorem 1.10 in \cite{Tenenbaum}. Thus we obtain
\begin{lemma}\label{bound on A(k)}
    For all $k\in\NN$,
    \begin{equation}\label{bound on A(k) eq}
        A(k)\leq\log(p_{k+1})\exp\left(c_1-\frac{1}{p_j}+\frac{5}{\log(p_{k+1})}\right).
    \end{equation}
\end{lemma}
\noindent Combining \cite{Dusart} and Theorem 3 from \cite{RosserShoenfeld}, we obtain the following:
\begin{theorem}\label{bounds on p_n}
    For $k\geq 6$,
    \begin{equation}\label{bounds on p_n eq}
        k(\log(k)+\log(\log(k))-1)<p_k<k(\log(k)+\log(\log(k))).
    \end{equation}
\end{theorem}
\noindent Furthermore, combining \ref{bound on A(k)} and \ref{bounds on p_n}, we see that
\begin{lemma}
    For $k\geq 6$, $A(k)<C(k)$ where
    \begin{equation}\label{def of C(k)}
        \begin{aligned}
            C(k) &= \log((k+1)(\log(k+1)+\log(\log(k+1)))) \\
            &\qquad \exp\left(c_1-\frac{1}{p_j}+\frac{5}{\log((k+1)(\log(k+1)+\log(\log(k+1))-1))}\right).
        \end{aligned}
    \end{equation}
\end{lemma}

Now, put $m=p_{k+1}\#/p_j$. Our goal is to show the following, since it implies that Robin's inequality for $n$ as above:
\begin{theorem}\label{C(k)<e^gamma log(log(m))}
    For any $j\in\{1,2,3\}$, there exists a $K_j\in\NN$ such that $k\geq K_j$ implies
    \begin{equation}
        C(k)B(k)<e^\gamma\log(\log(m)).
    \end{equation}
\end{theorem}

\begin{corollary}
    Suppose \ref{C(k)<e^gamma log(log(m))} holds. Then Robin's inequality holds for $n$ as above.
\end{corollary}

\begin{proof}
    We calculate
    \begin{equation}
        \frac{\sigma(n)}{n}<\frac{n}{\ph(n)}\leq A(k)B(k)<C(k)B(k)<e^\gamma\log(\log(m))\leq e^\gamma\log(\log(n)),
    \end{equation}
    where the last inequality follows from the fact that $m\leq n$.
\end{proof}

\begin{definition}
    The \emph{Chebyshev function} is defined as follows:
    \begin{equation}
        \theta(x)=\sum_{p\in\PP,p\leq x}\log(p)=\log\left(\prod_{p\in\PP,p\leq x}p\right).
    \end{equation}
\end{definition}

\begin{theorem}\label{bound on product of primes <=x}
    For $x\geq 529$,
    \begin{equation}
        \prod_{\substack{p\in\PP\\p\leq x}}p=e^{\theta(x)}>e^{x\left(1-\frac{1}{2\log x}\right)}\geq(2.51)^x.
    \end{equation}
\end{theorem}

\begin{proof}
    The first inequality is given by (3.14) in \cite{RosserShoenfeld} and the second follows from computations since the function $f(x)=1-\frac{1}{2\log x}$ increases for $x>1$. 
\end{proof}

\begin{lemma}\label{bound on log(log(m))}
    For $k\geq 99$,
    \begin{equation*}
        \log(\log(m))>\log((k+1)(\log(k+1)+\log(\log(k+1))-1)\log(2.51)-\log(p_j))=:D(k). 
    \end{equation*}
\end{lemma}

\begin{proof}
    Noting that $k\geq 99$ implies $p_{k+1}>529$, we calculate
    \begin{equation}
        \begin{aligned}
            \log(\log(m)) &= \log\left(\log\left(\frac{p_{k+1}\#}{p_j}\right)\right)>\log\left(\log\left(\frac{(2.51)^{p_{k+1}}}{p_j}\right)\right) \\
            &=  \log(p_{k+1}\log(2.51)-\log(p_j))\\
            &> \log((k+1)(\log(k+1)+\log(\log(k+1))-1)\log(2.51)-\log(p_j)),
        \end{aligned}
    \end{equation}
    where the last inequality uses \ref{bounds on p_n}.
\end{proof}

The following implies \ref{C(k)<e^gamma log(log(m))}:

\begin{proposition}
    For $j\in\{1,2,3\}$, there exists a $K_j\in\NN$ such that $k\geq K_j$ implies
    \begin{equation}\label{C(k)B(k)<e^gamma D(k)}
        C(k)B(k)<e^\gamma D(k).
    \end{equation}
\end{proposition}

\begin{proof}
    Denote 
    \begin{equation}
        \begin{aligned}
            \widetilde{C}(k) = e^{-c_1+\frac{1}{p_j}}C(k) &= \log((k+1)(\log(k+1)+\log(\log(k+1)))) \\
            &\qquad \exp\left(\frac{5}{\log((k+1)(\log(k+1)+\log(\log(k+1))-1))}\right) 
        \end{aligned}
    \end{equation}
    and
    \begin{equation}
        \widehat{C}(k)=\exp\left(\frac{5}{\log((k+1)(\log(k+1)+\log(\log(k+1))-1))}\right).
    \end{equation}
    Multiplying both sides of \eqref{C(k)B(k)<e^gamma D(k)} by $e^{-c_1+\frac{1}{p_j}}p_j^2/(p_j^2-1)$, we obtain
    \begin{equation}\label{C(k)B(k)<e^gamma D(k) proof eq 1}
        \widetilde{C}(k)\prod_{\ell=1}^{k+1}\frac{p_\ell^2}{p_\ell^2-1}<\frac{e^{\gamma-c_1+\frac{1}{p_j}p_j^2}}{p_j^2-1}D(k).
    \end{equation}
    Noting that
    \begin{equation}
        \prod_{\ell=1}^{k+1}\frac{p_\ell^2}{p_\ell^2-1}\nearrow\frac{\pi^2}{6}\text{ as }k\to\infty,
    \end{equation}
    we see that \eqref{C(k)B(k)<e^gamma D(k) proof eq 1} is implied by
    \begin{equation}\label{C(k)B(k)<e^gamma D(k) proof eq 2}
        \widetilde{C}(k)<\frac{6p_j^2e^{\gamma-c_1+\frac{1}{p_j}}}{\pi^2(p_j^2-1)}D(k)=:E_jD(k).
    \end{equation}
    Raising both sides to the power of $e$, we see that \eqref{C(k)B(k)<e^gamma D(k) proof eq 2} is implied by
    \begin{equation}\label{C(k)B(k)<e^gamma D(k) proof eq 3}
        \begin{aligned}
            & [(k+1)(\log(k+1)+\log(\log(k+1))]^{\widehat{C}(k)} \\
            &< [(k+1)(\log(k+1)+\log(\log(k+1))-1)\log(2.51)-\log(p_j)]^{E_j}.
        \end{aligned}
    \end{equation}
    \eqref{C(k)B(k)<e^gamma D(k) proof eq 3} is equivalent to
    \begin{equation}\label{C(k)B(k)<e^gamma D(k) proof eq 4}
        \begin{aligned}
            1 &< [(k+1)(\log(k+1)+\log(\log(k+1)))]^{-\widehat{C}(k)+E_j} \\
            &\qquad \left[1-\frac{(k+1)\log(2.51)+\log(p_j)}{(k+1)(\log(k+1)+\log(\log(k+1)))}\right]^{E_j}.
        \end{aligned} 
    \end{equation}
    Noting that $E_j>1$ for $j\in\{1,2,3\}$, we see that there exists a $K_j\in\NN$ such that $k\geq K_j$ implies $-\widehat{C}(k)+E_j>\epsilon$ for some $\epsilon\in(0,1)$. If needed, we can increase $K_j$ so that $k\geq K_j$ implies
    \begin{equation}
        \left[1-\frac{(k+1)\log(2.51)+\log(p_j)}{(k+1)(\log(k+1)+\log(\log(k+1)))}\right]^{E_j}>\epsilon,
    \end{equation}
    and also so that $k\geq K_j$ implies
    \begin{equation}
        1<\epsilon[(k+1)(\log(k+1)+\log(\log(k+1))]^\epsilon,
    \end{equation}
    which implies \eqref{C(k)B(k)<e^gamma D(k) proof eq 4}.
\end{proof}

\subsection{All numbers not divisible by one of the prime numbers 2,3, 5}

Letting $j=1$ in \eqref{C(k)B(k)<e^gamma D(k) proof eq 2}, we seek to show that
\begin{equation}\label{C(k)B(k)<e^gamma D(k) proof eq 5}
    \widetilde{C}(k)<\frac{8e^{\gamma-c_1+.5}}{\pi^2}D(k).
\end{equation}

\begin{lemma}\label{widehat C(k)<1.525}
    For $k\geq 13042$, $\widehat{C}(k)<1.525$.
\end{lemma}

\begin{proof}
    $\widehat{C}(k)$ is decreasing, so the result follows from computation.
\end{proof}

\noindent Denote $f(k)=(k+1)(\log(k+1)+\log(\log(k+1))$. Applying \ref{widehat C(k)<1.525}to \eqref{C(k)B(k)<e^gamma D(k) proof eq 5} and performing some algebraic manipulations, our goal reduces to showing that
\begin{equation}\label{C(k)B(k)<e^gamma D(k) proof eq 6}
    \log(f(k))<\frac{8e^{\gamma-c_1+.5}}{\pi^2(1.525)}\log((f(k)-1)\log(2.51)-\log(2)).
\end{equation}
Raising both sides to the power of $e$, this becomes
\begin{equation}\label{C(k)B(k)<e^gamma D(k) proof eq 7}
    1<f(k)^{2.0166}\left[1-\frac{(k+1)\log(2.51)-\log(2)}{f(k)}\right]^{1.20166}.
\end{equation}
The RHS of \eqref{C(k)B(k)<e^gamma D(k) proof eq 7} is increasing, and a computation reveals that it holds for $k\geq 13042$. Additionally, using \ref{sigma(n)/n<n/phi(n)}, one can check that
\begin{equation}
    \frac{\sigma(n)}{n}<\frac{n}{\ph(n)}<e^\gamma\log(\log(m))
\end{equation}
for $k\geq 3$. Finally, when $k\in\{1,2\}$, we check that
\begin{equation}\label{C(k)B(k)<e^gamma D(k) proof eq 8}
    \frac{\sigma(n)}{n}<\frac{n}{\ph(n)}\leq\frac{15}{8}<e^\gamma\log(\log(n))
\end{equation}
for $n\geq 680$. This confirms the following for $j=1$:
\begin{theorem}\label{Robin's inequality for numbers not divisible by p_j}
    For $j\in\{1,2,3\}$, Robin's inequality holds for every natural number $>5040$ which is not divisible by $p_j$.
\end{theorem}
\noindent To confirm \ref{Robin's inequality for numbers not divisible by p_j} when $j\in\{2,3\}$, one can repeat the above process to see that, for sufficiently big $k$, \eqref{C(k)B(k)<e^gamma D(k) proof eq 2} is satisfied. The cases with smaller $k$ have been verified in \cite{MorrillandPlatt}.

\subsection{Primorials and sufficiently big even numbers}
  We consider numbers of the form $2^{k}n$ for odd $n$.

Fix $k\in\NN$ and let $n$ be odd. We calculate
\begin{equation}
    \frac{\sigma(2^kn)}{2^kn}=\frac{\sigma(2^k)}{2^k}\frac{\sigma(n)}{n}<\frac{\sigma(2^k)}{2^k}\frac{n}{\ph(n)}=\frac{\sigma(2^k)}{2^k}\frac{\ph(2^k)}{2^k}\frac{2^kn}{\ph(2^kn)}=\left(1-\frac{1}{2^{k+1}}\right)\frac{2^kn}{\ph(2^kn)}.
\end{equation}
Applying Theorem 15 from \cite{RosserShoenfeld}, we know
\begin{equation}
    \left(1-\frac{1}{2^{k+1}}\right)\frac{2^kn}{\ph(2^kn)}<\left(1-\frac{1}{2^{k+1}}\right)\left(e^\gamma \log(\log(2^kn))+\frac{2.51}{\log(\log(2^kn))}\right).
\end{equation}
We ask which $n$ satisfy
\begin{equation}
    \left(1-\frac{1}{2^{k+1}}\right)\left(e^\gamma \log(\log(2^kn))+\frac{2.51}{\log(\log(2^kn))}\right)<e^\gamma(\log(\log(2^kn))).
\end{equation}
This is equivalent to asking when
\begin{equation}
    \frac{2.51(2^{k+1}-1)}{e^\gamma}<(\log(\log(2^kn)))^2
\end{equation}
holds, which is when
\begin{equation}\label{lower bound on n for sigma(2^k n)/2^k n<e^gamma log(log(2^k n))}
    n>\frac{e^{e^{\sqrt{\frac{2.51(2^{k+1}-1)}{e^\gamma}}}}}{2^k}=:b(k).
\end{equation}
Thus, we obtain the following:
\begin{theorem}\label{Robin's inequality holds for sufficiently big numbers of the form 2^k n}
    Given any $k\in\NN$, Robin's inequality holds for all numbers of the form $2^kn$ when $n$ is odd and satisfies \eqref{lower bound on n for sigma(2^k n)/2^k n<e^gamma log(log(2^k n))}.
\end{theorem}
In particular, we have
\begin{corollary}
    If $n\geq 620$ is odd, then Robin's inequality holds for $2n$. Furthermore, Robin's inequality holds for all primorials $>30$.
\end{corollary}

\begin{proof}
    The first statement follows immediately from \ref{Robin's inequality holds for sufficiently big numbers of the form 2^k n} and the second follows from the computation of primorials $<1240$.
\end{proof}

\subsection{All 21-free numbers}\label{21-free numbers}
  The results of the previous subsection, are based on the inequality in Theorem 15 from \cite{RosserShoenfeld}. This inequality can be improved by using a better bound stated in \cite{AxlerNicolas}:
    \begin{equation}
        \frac{m}{\ph(m)} < e^{\gamma} \log(\log(m))  + \frac{.0168}{(\log(\log(m)))^2}.
    \end{equation}
  for $m \geq 10^{10^{13.11485}}= C.$
    Using the same reasoning as before, we derive the following result.
    \begin{theorem}\label{Robin's inequality holds for sufficiently big numbers of the form 2^k n improved}
      Given any $k$ natural number, Robin's inequality holds for all numbers of the form $2^kn$ when $n$ is odd and satisfies 
      \begin{equation} 
         2^kn > e^{e^{\left(\frac{.0168(2^{k+1} -1)}{e^{\gamma}}\right)^{\frac{1}{3}}}} =: 2^k\tilde{b}(k)
      \end{equation}
    \end{theorem} 
     
  Furthermore, it was shown in \cite{MorrillandPlatt} that Robin's inequality holds for all natural numbers $5041 < m \leq C.$. We can thus conclude the following: 
   \begin{theorem}
      Robin's inequality holds for all natural numbers of the form $2^kn$ with $n$ odd as long as $2^k\tilde{b}(k) < C.$ 
      In particular, Robin's inequality holds for all 21-free numbers. 
    \end{theorem}
    \begin{proof}
      Let $k$ be a natural number and $n$ be an odd natural number. 
      \begin{itemize}
          \item if $5041< 2^kn \leq 2^k \tilde{b}(k)<C$ then $2^kn$ satisfies Robin's inequality by \cite{MorrillandPlatt}
          \item Alternatively, if $2^kn > 2^k\tilde{b}(k)$ then $2^kn$ satisfies Robin's inequality by \ref{Robin's inequality holds for sufficiently big numbers of the form 2^k n improved}.
      \end{itemize}
      
      Recalling that a $\ell$-free number is a natural number not divisible by any $\ell$ power of a prime number greater than or equal to $2$, we can see that if $2^k\tilde{b}(k)< C$ then all $(k+1)$-free numbers satisfy Robin's inequality.
      Since $\log(2^{20}\tilde{b}(20)) < 6(10^{11})< 2.3 (10^{13.11485} < \log(C)$, we can conclude that Robin's inequality holds for all 21-free numbers.
      \end{proof}

    \begin{remark}
        The validity of Robin's inequality for $\ell$-free numbers was proved for $\ell = 7$  in \cite{SolePlanat}, for $\ell = 11$ in \cite{BroughanTrudgian} and for $\ell =20$ in \cite{MorrillandPlatt}.
    \end{remark}

\subsection{Almost every number}\label{almost every number}

\begin{definition}
    The \emph{natural density} of a set $E$ is
    \begin{equation}
        d(E)=\lim_{s\to\infty}\frac{\#E\cap\{1,2,\dots,s\}}{s}
    \end{equation}
    when the limit exists.
\end{definition}

\begin{theorem}
    Denote by $\mathcal R$ the set of numbers satisfying Robin's inequality. Then the natural density of $\mathcal R$ is $1$.
\end{theorem}

\begin{proof}
    We will prove that the natural density of $\mathcal R^c$ is $0$. Fix $\epsilon>0$. Let $E_k=\{2^kn:n\in\mathbb N_\text{odd},n\leq b(k)\}$ and note that $\mathcal R^c\subseteq \bigcup_{k\geq 1}E_k$ by \ref{Robin's inequality for numbers not divisible by p_j} and \ref{Robin's inequality holds for sufficiently big numbers of the form 2^k n}.\footnote{Here $\mathbb N_\text{odd}:=\{1,3,\dots\}$.} Pick $M$ so that $\sum_{k=M+1}^\infty\frac{1}{2^k}<\frac\epsilon2$. For $s\in\mathbb N$ we calculate
    \begin{equation}\label{density calculations eq 1}
        \begin{aligned}
            \frac{\#\mathcal R^c\cap\{1,2,\dots,s\}}{s} &\leq \frac{\#\bigcup_{k\geq 1}E_k\cap\{1,2,\dots,s\}}{s}= \frac{\sum_{k\geq 1}\#E_k\cap\{1,2,\dots,s\}}{s} \\
            &= \frac{\sum_{k=1}^M \#E_k\cap\{1,2,\dots,s\}+\sum_{k=M+1}^\infty \#E_k\cap\{1,2,\dots,s\}}{s},
        \end{aligned}
    \end{equation}
    where the first equality follows from the fact that the $E_k$'s are disjoint. Noting that $\sum_{k=1}^M\#E_k\cap\{1,2,\dots,s\}<\infty$ for all $s\in\NN$, we see that we can pick $S$ so that $s\geq S$ implies that the RHS of \eqref{density calculations eq 1} is $<\epsilon$, completing our proof.
\end{proof}

\begin{remark}
    Using less elementary results from number theory, Wójtowicz \cite{Wojtowicz07} proves the following stronger result: for every $0 < C \le 1$, there exists a subset $W_C \subseteq \mathcal R$ such that $d(W_C)=1$ and for all $n\in W_C$,
    \begin{equation}
        \frac{\sigma(n)}{n}<Ce^\gamma \log(\log(n)).
    \end{equation}
\end{remark}

\section{The Lagarias and Kaneko-Lagarias Inequalities}\label{Lagarias ineq sec}

\subsection{Superabundant numbers}

Let $\Gamma(x)$ denote the gamma function. We define two functions:
\begin{equation}
    \begin{aligned}
        H(x) &= \int_0^1\frac{t^x-1}{t-1}\dd t, \\
        \psi(x) &= \frac{\Gamma'(x)}{\Gamma(x)}.
    \end{aligned}
\end{equation}
$\psi$ is known as the \emph{digamma function}. It is known that $H$ is smooth for $x\geq 1$ and that $H(n)=H_n$ for all $n\in\NN$. Additionally, $\psi$, known as the digamma function, satisfies
\begin{equation}\label{H(x)=psi(x+1)+gamma}
    H(x)=\psi(x+1)+\gamma.
\end{equation}

\begin{lemma}\label{first H(x) bound}
    For all $x\geq 1$,
    \begin{equation}\label{first H(x) bound eq}
        H(x)<\log(x)+\gamma+\frac{1}{2x}.
    \end{equation}
\end{lemma}

\begin{proof}
    By (2.2) from \cite{Alzer},
    \begin{equation}
        \psi(x)<\log(x)-\frac{1}{2x}
    \end{equation}
    for all $x\geq 1$. Then we use \eqref{H(x)=psi(x+1)+gamma} and $\psi(x+1)=\psi(x)+\frac 1x$ to finish.
\end{proof}

\begin{lemma}
    For all $x\geq 4$,
    \begin{equation}
        H(x)<\frac{2\log(x)}{1+\frac{6}{\pi^2 x}}.
    \end{equation}
\end{lemma}

\begin{proof}
    By \ref{first H(x) bound}, it suffices to show that
    \begin{equation}\label{second H(x) bound eq1}
        \log(x)+\gamma+\frac{1}{2x}<\frac{2\log(x)}{1+\frac{6}{\pi^2 x}}
    \end{equation}
    for $x\geq 4$. By arithmetic manipulations, \eqref{second H(x) bound eq1} becomes
    \begin{equation}\label{second H(x) bound eq2}
        \frac{1}{\pi^2 x-6}\left(\gamma\pi^2x+\frac{\pi^2}{2}+6\gamma+\frac 3x\right)<\log(x).
    \end{equation}
    Computation reveals that \eqref{second H(x) bound eq2} holds for $x=4$, and the LHS of \eqref{second H(x) bound eq2} is decreasing while the RHS is increasing, so we obtain the result.
\end{proof}

  \begin{lemma}\label{two estimates}
    The following hold:
    \begin{enumerate}
    \item For all $n>1$, $H_{n+1}\leq\frac{n}{\log(n)}$.
    \item For all $x\geq 4$, $\log(H(x))\leq\frac{x}{2\log(x)}$.
    \end{enumerate}
\end{lemma}

\begin{proof}
    (a) We can manually verify the inequality for $n\leq 6$. Noting that
    \begin{equation}
        H_{n+1}=\sum_{k=1}^{n+1}\frac 1k\leq 1+\int_1^{n+1}\frac{\dd t}{t}=1+\log(n+1),
    \end{equation}
    it suffices to show that
    \begin{equation}
        \log(x)(\log(x+1)+1)\leq x.
    \end{equation}
    Put $g(t)=e^t-t^2-t-1$. We see that $g(2)>0$ and that $g'(t)=e^t-2t-1>0$ for $t\geq 2$, so $g(t)>0$ for $t\geq 2$. For $x\geq e^2-1$ we have
    \begin{equation}
        0<g(\log(x+1))=x+1-(\log(x+1))^2-\log(x+1)-1<x-\log(x)(\log(x+1)+1).
    \end{equation}

    (b) For $x\geq 4$, note that the function mapping $x\mapsto \frac{x}{\log(x)}$ is increasing. If $n\leq x<n+1$, then
    \begin{equation}\label{bound on log(H(x)) eq}
        H_n\leq H(x)< H_{n+1}\leq\frac{n}{\log(n)}\leq\frac{x}{\log(x)}.
    \end{equation}
    For $y>2$ we see that $\log(y)<\frac y2$, so let $y=H(x)$ and apply \eqref{bound on log(H(x)) eq} finish.
\end{proof}

\begin{lemma}\label{H(x)log(H(x) bound}
    For $x\geq 4$,
    \begin{equation}
        H(x)\log(H(x))<\frac{x^2}{x+\frac{6}{\pi^2}}.
    \end{equation}
\end{lemma}

\begin{proof}
    Apply \ref{first H(x) bound} and \ref{two estimates}. 
\end{proof}

\begin{lemma}\label{H'(x) bound}
    For $x\geq 4$, 
    \begin{equation}
        H'(x)>\frac{H(x)\log(H(x))}{x^2}.
    \end{equation}
\end{lemma}

\begin{proof}
    We will use (51) from \cite{DolatabadiandFarhanddoost} which states that
    \begin{equation}
        \frac{1}{\psi'(x)}\leq x+\frac{6}{\pi^2}-1
    \end{equation}
    for $x\geq 1$. We calculate
    \begin{equation}
        H'(x)=\psi'(x+1)\geq\frac{1}{x+6\pi^2}>\frac{H(x)\log(H(x))}{x^2},
    \end{equation}
    where the equality follows from taking the derivative of \eqref{H(x)=psi(x+1)+gamma} and the second inequality follows from \ref{H(x)log(H(x) bound}.
\end{proof}

\begin{proposition}\label{f(x) increasing}
    The function
    \begin{equation}
        g(x)=\frac{\exp(H(x))\log(H(x))}{x}
    \end{equation}
    is increasing for $x\geq 4$.
\end{proposition}

\begin{proof}
    We start with (3.5) from \cite{Lagarias}:
    \begin{equation}
        \begin{aligned}
            & H_n=\log(n)+\gamma+\int_n^\infty \frac{x-\lfloor x\rfloor}{x^2}\dd x \\
            &\implies \exp(H_n)=e^\gamma n\exp\left(\int_n^\infty\frac{x-\lfloor x\rfloor}{x^2}\dd x\right) \\
            &\implies \frac{\exp(H_n)\log(H_n)}{n}=e^\gamma\log(H_n)\exp\left(\int_n^\infty\frac{x-\lfloor x\rfloor}{x^2}\dd x\right).
        \end{aligned}
    \end{equation}
    Given $k\in\NN$, put
    \begin{equation}
        g_k(x)=e^\gamma\log(H(x))\exp\left(\int_x^k\frac{t-\lfloor t\rfloor}{t^2}\dd t\right)
    \end{equation}
    so that $\lim_{k\to\infty}g_k(x)=g(x)$. We compute
    \begin{equation}
        g_k'(x)=e^\gamma\exp\left(\int_x^k\frac{t-\lfloor t\rfloor}{t^2}\dd t\right)\left(\frac{H'(x)}{H(x)}+\log(H(x))\left(-\frac{x-\lfloor x\rfloor}{x^2}\right)\right),
    \end{equation}
    so $g'_k(x)>0$ if and only if
    \begin{equation}
        \frac{H'(x)}{H(x)}+\log(H(x))\left(-\frac{x-\lfloor x\rfloor}{x^2}\right)\geq \frac{H'(x)}{H(x)}-\frac{\log(H(x))}{x^2}>0,
    \end{equation}
    which is the content of \ref{H'(x) bound}. Thus, $g(x)$ is the limit of monotonically increasing functions and is therefore monotonically increasing.
\end{proof}

\begin{corollary}
    The sequence
    \begin{equation}
        \left\{\frac{\exp(H_n)\log(H_n)}{n}\right\}_{n=1}^\infty
    \end{equation}
    is monotonically increasing.
\end{corollary}

\begin{proof}
    \ref{f(x) increasing} gives the result for $n\geq 4$ and we can manually check the smaller cases.
\end{proof}

\begin{definition}
    A number $n$ is \emph{superabundant} if $\sigma(m)/m<\sigma(n)/n$ for all $m<n$.
\end{definition}

\begin{theorem}\label{LLI theorem}
    If there are counterexamples to the Kaneko-Lagarias inequality, the smallest such counterexample is a superabundant number.
\end{theorem}

\begin{proof}
    Suppose, for sake of contradiction, that $m$ is the smallest counterexample to the Kaneko-Lagarias inequality and that $m$ is not superabundant. Let $n$ be the greatest superabundant number $<m$. We calculate,
    \begin{equation}
        \frac{\sigma(n)}{n}>\frac{\sigma(m)}{m}\geq\frac{\exp(H_m)\log(H_m)}{m}>\frac{\exp(H_n)\log(H_n)}{n},
    \end{equation}
    so $n<m$ violates the Kaneko-Lagarias inequality: a contradiction.
\end{proof}

\subsection{Connection to Robin's inequality}

\begin{theorem}
    If Robin's inequality holds for some $n\in\NN$, then the Kaneko-Lagarias inequality holds for $n$.
\end{theorem}

\begin{proof}
    We use the approximation
\begin{equation}
        H_n\geq\log(n)+\gamma+\frac{1}{2n+1}
    \end{equation}
    to calculate
    \begin{equation}
        \begin{aligned}
            \frac{\exp(H_n)\log(H_n)}{n}\geq \frac{e^{\gamma+\frac{1}{2n+1}}n\log\left(\log(n)+\gamma+\frac{1}{2n+1}\right)}{n}>e^\gamma\log(\log(n)),
        \end{aligned}
    \end{equation}
    which implies the result.
\end{proof}

Note that we obtain the same result for the Lagarias inequality.

\section*{Acknowledgments}

The first author thanks Jeff Lagarias for his comments and the references he provided. The third author thanks Keith Briggs for providing his code used to compute superabundant numbers, Perry Thompson and Owen McAllister for their help implementing it in Rust and Jean-Louis Nicolas for sharing the paper \cite{AxlerNicolas}.We also thank the anonymous referees for their helpful comments, which certainly improved the quality of this paper.

\printbibliography

@book {Tenenbaum,
        AUTHOR = {Tenenbaum, G\'erald},
        TITLE = {Introduction to analytic and probabilistic number theory},
        SERIES = {Cambridge Studies in Advanced Mathematics},
        VOLUME = {46},
        EDITION = {French},
        PUBLISHER = {Cambridge University Press, Cambridge},
        YEAR = {1995},
        PAGES = {xvi+448},
    }

@article{RosserShoenfeld,
        author = {J. Barkley  Rosser and Lowell Schoenfeld},
        title = {{Approximate formulas for some functions of prime numbers}},
        volume = {6},
        journal = {Illinois Journal of Mathematics},
        number = {1},
        publisher = {Duke University Press},
        pages = {64 -- 94},
        year = {1962},
    }

@article{Dusart,
        author = {Pierre Dusart},
        journal = {Mathematics of Computation},
        number = {225},
        pages = {411--415},
        publisher = {American Mathematical Society},
        title = {The $k$-th Prime is Greater than $k(\log(k)() + \log(\log(k)) - 1)$ for $k\geq 2$},
        volume = {68},
        year = {1999}
    }

@article{Lagarias,
        author = {Jeffrey C. Lagarias},
        journal = {The American Mathematical Monthly},
        number = {6},
        pages = {534--543},
        publisher = {Mathematical Association of America},
        title = {An Elementary Problem Equivalent to the Riemann Hypothesis},
        volume = {109},
        year = {2002}
    }

@article{Robin,
        author = {Guy Robin},
        journal = {Journal de Mathématiques Pures et Appliquées},
        volume = {63},
        pages = {187--213},
        year = {1984},
        title = {Grandes valeurs de la fonction somme des diviseurs et hypothèse de Riemann},
    }

@article{Choieetal,
        author = {YoungJu Choie and Nicolas Lichiardopol and Pieter Moree and Patrick Solé},
        journal = {Journal de Théorie des Nombres de Bordeaux},
        number = {2},
        pages = {357--372},
        publisher = {Société Arithmétique de Bordeaux},
        title = {On Robin's criterion for the Riemann hypothesis},
        volume = {19},
        year = {2007}
    }

@article{AkbaryFriggstad,
        author = {Amir Akbary and Zachary Friggstad},
        journal = {The American Mathematical Monthly},
        number = {3},
        pages = {273--275},
        publisher = {[Taylor & Francis, Ltd., Mathematical Association of America]},
        title = {Superabundant Numbers and the Riemann Hypothesis},
        volume = {116},
        year = {2009}
    }

@article{Alzer,
        author = {Horst Alzer},
        journal = {Mathematics of Computation},
        number = {217},
        pages = {373--389},
        publisher = {American Mathematical Society},
        title = {On Some Inequalities for the Gamma and Psi Functions},
        volume = {66},
        year = {1997}
    }

@ARTICLE{DolatabadiandFarhanddoost,
        title = {New Inequalities for Gamma and Digamma Functions},
        author = {Farhangdoost, M. R. and Dolatabadi, M. Kargar},
        year = {2014},
        journal = {Journal of Applied Mathematics},
        volume = {2014},
        pages = {1-7},
    }

@article{MorrillandPlatt,
          title = {Robin's inequality for 20-free integers}, 
          author={Thomas Morrill and David Platt},
          year={2020},
          journal = {Integers},
           volume = {21},
           
    }

@article{AxlerNicolas, 
          title = {Large values of $\frac{n}{\ph(n)}$},
          author = {Christian Axler and Jean-Louis Nicolas},
          year = {2022},
          journal = {Acta Arithmetica},
          volume = {209},
          pages = {357-383},
    }

@article{SolePlanat,
        title = {The Robin inequality for 7-free integers},
        author = {P. Sole and M. Planat}, 
        year = {2012},
        journal = {Integers},
        Volume = {12},
        pages = {301-309},    
    }

@article{BroughanTrudgian, 
        title = {Robin's inequality for 11-free integers},
        author = {Kevin Boughan and Tim Trudgian} ,
        year = {2015},
        journal = {Integers},
        volume = {15},    
    }

@article{Axler23,
      title={On Robin’s inequality},
      author={Axler, Christian},
      journal={The Ramanujan Journal},
      volume={61},
      number={3},
      pages={909--919},
      year={2023},
      publisher={Springer}
    }

@article{Hertlein16,
      title={Robin's inequality for new families of integers},
      author={Hertlein, Alexander},
      journal={arXiv preprint arXiv:1612.05186},
      year={2016}
    }

@article{AFJ07,
      title={Explicit upper bounds for $\prod_{p_\omega(n)}\frac{p}{p-1}$},
      author={Akbary, Amir and Friggstad, Zachary and Juricevic, Robert},
      journal={Contrib. Discrete Math.},
      volume={2},
      number={2},
      pages={153--160},
      year={2007}
    }

@article{Wojtowicz07,
      title={Robin's inequality and the Riemann hypothesis},
      author={W{\'o}jtowicz, Marek},
      journal={Proceedings of the Japan Academy of Sciences},
      volume={83},
      number={4},
      year={2007}
    }

@article{WashingtonYang21,
      title={Analogues of the Robin--Lagarias criteria for the Riemann hypothesis},
      author={Washington, Lawrence C and Yang, Ambrose},
      journal={International Journal of Number Theory},
      volume={17},
      number={04},
      pages={843--870},
      year={2021},
      publisher={World Scientific}
    }

@article{Luca,
      author = {Luca, Florian and Pomerance, Carl and Sol\'e, Patrick},
      title = {Corrigendum to \#A8 of Volume 25, on Robin's Inequality},
      journal = {Integers},
      year = {2025},
      volume = {25},
    }

\end{document}